\documentclass[12pt]{amsart}

\usepackage{booktabs,url}
\usepackage{amsmath,mathtools,amsthm,amsfonts,amssymb,verbatim,cases,graphicx}
\usepackage[table,usenames,dvipsnames]{xcolor}
\usepackage[letterpaper]{geometry}
\usepackage{stmaryrd}
\usepackage{tabularx}

\theoremstyle{definition}
\newtheorem{theorem}[equation]{Theorem}
\newtheorem{corollary}[equation]{Corollary}

\newtheorem{proposition}[equation]{Proposition}

\newtheorem{definition}[equation]{Definition}

\newtheorem{rem}[equation]{Remark}
\newtheorem{exam}[equation]{Example}
\numberwithin{equation}{section}

\definecolor{mjo}{rgb}{0,0,.9}
\newcommand{\ZZ}{\mathbb{Z}}

\newcommand{\AGL}{\operatorname{AGL}}
\newcommand{\GL}{\operatorname{GL}}

\global\long\def\gen{\text{\sf GEN}}
\global\long\def\dng{\text{\sf DNG}}
\global\long\def\mex{\operatorname{mex}}
\global\long\def\nim{\operatorname{nim}}
\global\long\def\opt{\operatorname{Opt}}

\global\long\def\pty{\operatorname{pty}}
\global\long\def\type{\operatorname{type}}
\global\long\def\otype{\operatorname{otype}}

\newcommand{\notes}[1]{}


\begin{document}
\setlength{\jot}{0pt} 

\title[Impartial games for generating symmetric and alternating groups]
{Impartial avoidance and achievement games for generating symmetric and alternating groups}

\author{Bret J.~Benesh}
\author{Dana C.~Ernst}
\author{N\'andor Sieben}

\address{
Bret Benesh,
Department of Mathematics,
College of Saint Benedict and Saint John's University,
37 College Avenue South,
Saint Joseph, MN 56374-5011, USA
}
\email{\url{bbenesh@csbsju.edu}}
\address{
Dana Ernst and N\'andor Sieben,
Department of Mathematics and Statistics,
Northern Arizona University,
PO Box 5717,
Flagstaff, AZ 86011-5717, USA
}
\email{\url{Dana.Ernst@nau.edu}, \url{Nandor.Sieben@nau.edu}}

\thanks{This work was conducted during the third author's visit to DIMACS partially enabled through support from the National Science Foundation under grant number \#CCF-1445755.}

\date{\today}


\begin{abstract}
Anderson and Harary introduced two impartial games on finite groups.  Both games are played by two players who alternately select previously-unselected elements of a finite group. The first player who builds a generating set from the jointly-selected elements wins the first game. The first player who cannot select an element without building a generating set loses the second game. We determine the nim-numbers, and therefore the outcomes, of these games for symmetric and alternating groups.
\end{abstract}


\maketitle


\begin{section}{Introduction}


Anderson and Harary~\cite{anderson.harary:achievement} introduced two impartial games \emph{Generate} and \emph{Do Not Generate}
in which two players alternately take turns selecting previously-unselected elements of a finite group $G$. 
The first player who builds a generating set for the group from the jointly-selected elements wins the achievement game $\gen(G)$. 
The first player who cannot select an element without building a generating set loses the avoidance game $\dng(G)$. 
The outcomes of both games were studied for some of the more familiar finite groups, including abelian, dihedral, symmetric, and alternating groups in~\cite{anderson.harary:achievement,Barnes}, 
although the analysis was incomplete for alternating groups.   

Brandenburg studies a similar game in \cite{brandenburg:algebraicGames}. This variation of the game is played on a finitely generated abelian group $G$. 
Players alternate replacing $G$ by the quotient $G / \langle g \rangle$ for a choice of a nontrivial element $g$. 
The player with the last possible move wins.

A fundamental problem in the theory of impartial combinatorial games is finding the nim-number of a game.  
The nim-number determines the outcome of the game and also allows for the easy calculation of the nim-numbers of game sums. The last two authors~\cite{ErnstSieben} developed tools for studying the nim-numbers of both games, which they 
applied in the context of certain finite groups including cyclic, abelian, and dihedral. In~\cite{BeneshErnstSiebenDNG}, we developed a complete set of criteria for determining the nim-numbers of avoidance games and calculated the corresponding
nim-numbers for several families of groups, including symmetric groups. 

The previous work in~\cite{BeneshErnstSiebenDNG,ErnstSieben} 
left open the determination of the nim-numbers of $\dng(A_n)$, $\gen(S_n)$, and $\gen(A_n)$, except for a handful of cases.  
In this paper, we provide a complete classification of the nim-numbers of the games involving the symmetric and alternating groups.  
Since nim-numbers determine the outcome, our classification completes the analysis from~\cite{Barnes}.  In particular, under optimal play, the first player will win the achievement game and lose the avoidance game for most symmetric and alternating groups.  Moreover, building on the work of~\cite{BeneshErnstSiebenDNG}, we lay the foundation for how one might ascertain the nim-numbers for finite simple groups.

The paper is organized as follows. In Section~\ref{sec:Preliminaries}, we recall the basic terminology of impartial games, establish our notation, and review several necessary results 
from~\cite{ErnstSieben}. The following section describes how to compute $\gen(G)$ for $G$ in a family called $\Gamma_1$, which includes the nonabelian alternating and symmetric groups,  as well as all other finite nonabelian simple groups.  In Section~\ref{sec:Symmetric}, we review the known results for $\dng(S_n)$ and compute the nim-numbers for $\gen(S_n)$, and then we find the nim-numbers for $\dng(A_n)$ and $\gen(A_n)$ in Section~\ref{sec:Alternating}. 
We close with some open questions. 

The authors thank the referee for making suggestions that improved the quality of this paper.

\end{section}


\begin{section}{Preliminaries}\label{sec:Preliminaries}


\begin{subsection}{Impartial Games}\label{subsec:Games}


In this section, we briefly recall the basic terminology of impartial games and establish our notation. For a comprehensive treatment of impartial games, see~\cite{albert2007lessons,SiegelBook}. 

An \emph{impartial game} is a finite set $X$ of \emph{positions} together with a starting position and a collection $\{\opt(P)\subseteq X\mid P\in X\}$, where $\opt(P)$ is the set of possible \emph{options} for a position $P$. Two players take turns replacing the current position $P$ with one of the available options in $\opt(P)$. A position $P$ with $\opt(P)=\emptyset$ is called \emph{terminal}.  
The player who moves into a terminal position \emph{wins.}
All games must come to an end in finitely many turns. 

The \emph{minimum excludant} $\mex(A)$ of a set $A$ of ordinals is the smallest ordinal not contained in the set. The \emph{nim-number} of a position $P$ is the minimum excludant 
\[
\nim(P):=\mex\{\nim(Q)\mid Q\in\opt(P)\}
\]
of the set of nim-numbers of the options of $P$. Since the minimum excludant of the empty set is $0$, the terminal positions of a game have nim-number $0$. The \emph{nim-number of a game} is the nim-number of its starting position. 
The nim-number of a game determines the outcome  since the second player has a winning strategy if and only if the nim-number for a game is $0$. The winning strategy is to always pick an option that has nim-number 0. This puts the opponent into a losing position at every turn.

The \emph{sum} of the games $P$ and $R$ is the game $P+R$ whose set of options is 
\[
\opt(P+R):=\{Q+R\mid Q\in\opt(P)\}\cup\{P+S\mid S\in\opt(R)\}.
\]
We write $P=R$ if the second player of the game $P+R$ has a winning strategy. 

The one-pile NIM game with $n$ stones is denoted by the \emph{nimber} $*n$. The set of options of $*n$ is $\opt(*n)=\{*0,\ldots,*(n-1)\}$.  The Sprague--Grundy Theorem~\cite{albert2007lessons,SiegelBook} states that $P=*\nim(P)$ for every impartial game $P$.

\end{subsection}


\begin{subsection}{Achievement and Avoidance Games for Generating Groups}\label{subsec:GamesOnGroups}


We now provide a more precise description of the achievement and avoidance games for generating a finite group $G$. 
For additional details, see~\cite{anderson.harary:achievement,Barnes,ErnstSieben}.   For both games, the starting position is the empty set.  The players take turns choosing previously-unselected elements to create jointly-selected sets of elements, which are the positions of the game.

In the avoidance game $\dng(G)$, the first player chooses $x_{1}\in G$ such that $\langle x_{1}\rangle\neq G$ and at the $k$th turn, the designated player selects $x_{k}\in G\setminus\{x_{1},\ldots,x_{k-1}\}$ such that $\langle x_{1},\ldots,x_{k}\rangle\neq G$ to create the position $\{x_1,\ldots,x_k\}$. Notice that the positions of $\dng(G)$ are exactly the non-generating subsets of $G$. A position $Q$ is an option of $P$ if $Q=P \cup \{g\}$ for some $g \in G \setminus P$, where $\langle Q\rangle\neq G$.  The player who cannot select an element without building a generating set is the loser.  We note that there is no avoidance game for the trivial group since the empty set generates the whole group.

In the achievement game $\gen(G)$, the first player chooses any $x_{1}\in G$ and at the $k$th turn, the designated player selects $x_{k}\in G\setminus\{x_{1},\ldots,x_{k-1}\}$ to create the position $\{x_1,\ldots,x_k\}$. A player wins on the $n$th turn if $\langle x_{1},\ldots,x_{n}\rangle=G$. In this case, the positions of $\gen(G)$ are subsets of terminal positions, which are certain generating sets of $G$.  
Note that the second player has a winning strategy if $G$ is trivial since $\langle \emptyset \rangle = G$, so the first player has no legal opening move.
In particular, $\gen(S_1)=\gen(A_1)=\gen(A_2)=*0$.

Maximal subgroups play an important role because a subset $S$ of a finite group is a generating set if and only if $S$ is not contained in any maximal subgroup. Let $\mathcal{M}$ be the set of maximal subgroups of $G$. Following~\cite{ErnstSieben}, the set
of \emph{intersection subgroups} is defined to be the set 
\[
\mathcal{I}:=\{\cap\mathcal{N}\mid\emptyset\not=\mathcal{N\subseteq\mathcal{M}}\}
\]
containing all possible intersections of maximal subgroups. Note that the elements of $\mathcal{I}$ are in fact subgroups of $G$. The smallest intersection subgroup is the Frattini subgroup $\Phi(G)$ of $G$, and we say that an intersection subgroup $I$ is non-Frattini if $I \not=\Phi(G)$. Not every subgroup of $G$ need be an intersection subgroup.  For instance, no proper subgroup of a nontrivial Frattini subgroup is an intersection subgroup. Even a subgroup that contains the Frattini subgroup need not be intersection subgroup, as shown in \cite[Example 1]{BeneshErnstSiebenDNG}.

The set $\mathcal{I}$ of intersection subgroups is partially ordered by inclusion.  For each $I \in \mathcal{I}$, we let 
\[
X_{I}:=\mathcal{P}(I)\setminus\cup\{\mathcal{P}(J)\mid J\in \mathcal{I},J<I\}
\]
\noindent
be the collection of those subsets of $I$ that are not contained in any other intersection subgroup smaller than $I$. We define $\mathcal{X}:=\{X_{I}\mid I\in\mathcal{I}\}$ and call an element of $\mathcal{X}$ a \emph{structure class}.  The starting position $\emptyset$ is in $X_{\Phi(G)}$ for both games.

The set $\mathcal{X}$ of structure classes is a partition of the set of game positions of $\dng(G)$.  The partition $\mathcal{X}$ is compatible with the option relationship between game positions~\cite[Corollary~3.11]{ErnstSieben}: if $X_{I},X_{J}\in\mathcal{X}$ and $P,Q\in X_{I}\ne X_{J}$, then $\opt(P)\cap X_{J}\not=\emptyset$ if and only if $\opt(Q)\cap X_{J}\ne\emptyset$.  We say that $X_{J}$ is an \emph{option} of $X_{I}$ and write $X_{J}\in\opt(X_{I})$ if $\opt(I)\cap X_{J}\not=\emptyset$.  

For the achievement game $\gen(G)$, we must include an additional structure
class $X_{G}$ containing terminal positions. A subset $S\subseteq G$ belongs to $X_{G}$ whenever $S$ generates $G$ while $S\setminus\{s\}$ does not for some $s\in S$. Note that we are abusing notation since $X_{G}$ may not contain $G$. The positions of $\gen(G)$ are the positions of $\dng(G)$ together with the elements of $X_{G}$.  The set $\mathcal{Y}:=\mathcal{X}\cup\{X_{G}\}$ is a partition of the set of game positions of $\gen(G)$. As in the avoidance case, the partition $\mathcal{Y}$ is compatible with the option relationship between positions~\cite[Corollary~4.3]{ErnstSieben}.

For any position $\{g\}$ of either $\dng(G)$ or $\gen(G)$, we define $\lceil g \rceil$ to be the unique element of $\mathcal{I}\cup\{G\}$ such that $\{g\}\in X_{\lceil g \rceil}$. 
For example, if $e$ is the identity of $G$, then $\lceil e \rceil=\Phi(G)$.

We define $\pty(S)$ of a set $S$ to be $0$ if $|S|$ is even and $1$ if $|S|$ is odd, and we say that $S$ is  \emph{even} if $\pty(S)=0$ and \emph{odd} if $\pty(S)=1$. The \emph{type} of the structure class $X_{I}$ is the triple 
\[
\type(X_{I}):=(\pty(I),\nim(P),\nim(Q)),
\]
where $P,Q\in X_{I}$ with $\pty(P)=0$ and $\pty(Q)=1$; by~\cite[Proposition~3.15]{ErnstSieben}, this is well-defined. The type can be calculated as in the following example:  
if an odd $X_I$ has options of type $(0,a,b)$ and $(1,c,d)$, then $\type(X_I)=(1,\mex\{b,d,x\},x)$, where $x=\mex\{a,c\}$.  In the achievement case, the type of the structure class $X_{G}$ is defined to be $\type(X_{G})=(\pty(G),0,0)$.  The \emph{option type} of $X_{I}$ is the set 
\[
\otype(X_{I}):=\{\type(X_{K})\mid X_{K}\in\opt(X_{I})\}.
\]

For both the avoidance and achievement games, the nim-number of the game is the same as the nim-number of the initial position $\emptyset$, which is an even subset of $\Phi(G)$. Because of this, the nim-number of the game is the second component of $\type(X_{\Phi(G)})$.

\end{subsection}


\begin{subsection}{Achievement and Avoidance Game Strategy and Intuition}


Let $G$ be a nontrivial finite group.  In general, describing strategies in terms of nim-numbers is difficult, however we are able to provide a full description of the relationship between the strategy and the nim-number of $\dng(G)$. A complete characterization is elusive for $\gen(G)$. 

After playing the avoidance game, the two players can notice that they spent the entire game choosing elements from a single maximal subgroup $M$.  The first player will have won if $M$ is odd, and the second player will have won if $M$ is even.  These facts immediately yield the winning strategies for the two players.  The second player should attempt to choose any legal element of even order since that will ensure a final maximal subgroup of even order, while the first player should try to select an element $g \in G$ such that $\langle g,t \rangle=G$ for any element $t \in G$ of even order.  Of course, only one of these strategies can be successful.  Note that a winning strategy for the first player is equivalent to choosing an element $g \in G$ that is not contained in any even maximal subgroup.    The nim-numbers of the avoidance game were characterized in the following theorem.

\begin{theorem}\cite[Theorem~6.3]{BeneshErnstSiebenDNG}
Let $G$ be a nontrivial finite group.
\begin{enumerate}
\item If $|G|=2$ or $G$ is odd, then $\dng(G)=*1$.
\item If $G$ is cyclic with $|G| \equiv_4 0$ or the set of even maximal subgroups covers $G$, then $\dng(G)=*0$.
\item Otherwise, $\dng(G)=*3$.
\end{enumerate}
\end{theorem}

It follows that a nim-number of $0$ indicates that there is always an element of even order for the second player to choose, regardless of the first player's initial choice.  
Notice that $G$ only has odd maximal subgroups if $|G|=2$ or $G$ is odd, so a nim-number of $1$ indicates that the first player cannot lose, independent of strategy.  
A nim-number of $3$ indicates that there is an element not contained in any even maximal subgroup, and choosing this element on the first move is sufficient to guarantee victory for the first player.  
However, the first player may lose by failing to choose this element.  
Thus, the difference between $\dng(G)=*1$ and $\dng(G)=*3$ is that the first player needs to make a wise first move to ensure a win only in the latter case.  

Now let $G$ be $S_n$ or $A_n$ for $n \geq 5$.  We provide a description of the strategy for $\gen(G)$, but different positive nim-numbers for $\gen(G)$ do not seem to translate to any obvious group-theoretic description.  The group $G$ has the property that for every nontrivial $g \in G$, there is an $x \in G$ such that $\langle g,x\rangle = G$.  The first player can win by choosing the identity with the first selection and an $x$ such that $\langle g,x\rangle=G$ for the third move after the second player chooses a nonidentity element $g$. 

\end{subsection}

\end{section}


\begin{section}{Achievement Games for $\Gamma_1$ Groups}


In this section we develop some general results about $\gen(G)$ for a special class of groups containing all of the nonabelian alternating groups and all but one of the nonabelian symmetric groups.

Following~\cite[Definition 1.1]{Foguel}, two elements $x$ and $y$ of a group $G$ are called \emph{mates} if $\langle x,y \rangle=G$.  We say that a finite nonabelian 
group is in $\Gamma_1$ if every nontrivial element of the group has a mate (see~\cite[Definition 1.01]{BrennerWiegoldI} 
for a general definition of $\Gamma_r$).  It is known that $S_n \in \Gamma_1$ for $n \not\in \{1,2,4\}$ ~\cite{isaacs1995generating} and $A_n \in \Gamma_1$ for $n \not\in\{1,2,3\}$~\cite{chigira1997generating}.  Moreover, Guralnick and Kantor~\cite{Guralnick2000} proved that every finite nonabelian simple group is in $\Gamma_1$.

For the remainder of this paper, we will rely on $\Gamma_1$ groups having trivial Frattini subgroups. 

\begin{proposition}\label{prop:GammaFrattini}
If $G \in \Gamma_1$, then $\Phi(G)=\{e\}$.
\end{proposition}

\begin{proof}
For sake of a contradiction suppose that $x$ is a nontrivial element of $\Phi(G)$.  Since $G \in \Gamma_1$, $x$ has a mate $y$.  Then $G=\langle x,y \rangle \leq \langle \Phi(G),y \rangle = \langle y \rangle$, where the last equality holds because $\Phi(G)$ is the set of all nongenerators of $G$ (see~\cite[Problem 2.7]{Isaacs1994}, for instance).  This implies $G$ is cyclic and hence abelian, which contradicts the assumption that $G \in \Gamma_1$. 
\end{proof}

The following general results about $\Gamma_1$ groups will be used to determine the nim-numbers of $\gen(S_n)$ and $\gen(A_n)$.
We define $T:=\{t_0,\ldots,t_4\}$, where the types $t_i$ are defined in Table~\ref{table:typetable}.

\begin{table}[h]
\begin{tabular}{@{}l@{}}
$t_0:=(\pty(G),0,0)$ \\
$t_1:=(1,1,2)$ \\
$t_2:=(1,2,1)$ \\
$t_3:=(1,4,3)$ \\
$t_4:=(0,1,2)$ \\
\end{tabular}
\hspace{1cm}
\begin{tabular}{@{}ll@{}}
\toprule
$\otype(X_I)$ & $\type(X_I)$ \\
\midrule
$\{t_0\}$ & $t_2$ or $t_4$ \\  
$\{t_0,t_1\}$ & $t_1$  \\ 
$\{t_0,t_2\}$ &  $t_2$ \\
$\{t_0,t_3\}$ &  $t_2$ \\
$\{t_0,t_1,t_2\}$ &  $t_3$ \\
$\{t_0,t_1,t_3\}$ &  $t_1$ \\
$\{t_0,t_2,t_3\}$ &  $t_2$ \\
$\{t_0,t_1,t_2,t_3\}$ &  $t_3$ \\
\bottomrule
\end{tabular}
\hspace{.25cm}
\begin{tabular}{@{}ll@{}}
\toprule
$\otype(X_I)$ & $\type(X_I)$ \\
\midrule
$\{t_0,t_4\}$ &  $t_1$ or $t_4$ \\
$\{t_0,t_1,t_4\}$ &  $t_1$ \\
$\{t_0,t_2,t_4\}$ &  $t_3$ \\
$\{t_0,t_3,t_4\}$ &  $t_1$ \\
$\{t_0,t_1,t_2,t_4\}$ &  $t_3$ \\
$\{t_0,t_1,t_3,t_4\}$ &  $t_1$ \\
$\{t_0,t_2,t_3,t_4\}$ &  $t_3$ \\
$\{t_0,t_1,t_2,t_3,t_4\}$ &  $t_3$ \\
\bottomrule
\end{tabular}

\vspace{0.1in}

\caption{\label{table:typetable}  
 Complete list of possible option types for a nonterminal structure class $X_I$ in $\gen(G)$ with $I \not=\Phi(G)$ when $G$ is a $\Gamma_1$ group.  Note that $\pty(I)$ and $\otype(X_I)$ determine $\type(X_I)$.
}
\end{table}

\begin{proposition}\label{prop:PossibleTypes}
If $I$ is a non-Frattini intersection subgroup of a $\Gamma_1$ group $G$, then $\type(X_I)$ in $\gen(G)$
satisfies
\[
\type(X_I) =  
   \begin{dcases} 
      t_1,  & \text{$I$ is odd, $\otype(X_I) \cap \{t_1,t_4\} \not= \emptyset$, $t_2 \not\in\otype(X_I)$} \\
      t_2,  & \text{$I$ is odd, $\otype(X_I) \cap \{t_1,t_4\} = \emptyset$} \\
      t_3,  & \text{$I$ is odd, $\otype(X_I) \cap \{t_1,t_4\} \not= \emptyset$, $t_2 \in\otype(X_I)$} \\
      t_4, & \text{$I$ is even}.
   \end{dcases}
\]  
\end{proposition}

\begin{proof}
The type of a terminal structure class must be $t_0$.  The option type of any nonterminal structure class $X_I$ of a $\Gamma_1$ group with $I \not=\Phi(G)$ contains $t_0$.  Structural induction and the calculations in Table~\ref{table:typetable} shows that $\otype(X_I)\subseteq T$ implies $\type(X_I)\in T$. That is, no types outside of $T$ arise.  Note that a structure class of type $t_4$ can never have an odd option by Lagrange's Theorem, so $t_4$ only appears if $\otype(X_I)$ is a subset of $\{t_0,t_4\}$.
\end{proof}

For odd $G$, only some of the calculations from Table~\ref{table:typetable} are needed since the only possible option types are $\{t_0\}$ and $\{t_0,t_2\}$.  This observation is sufficient to determine $\gen(G)$ for odd $G$. 

\begin{corollary}\label{cor:OddType}
If $I$ is a non-Frattini intersection subgroup of an odd $\Gamma_1$ group $G$, then $\type(X_I)$ in $\gen(G)$ is $t_2$.
\end{corollary}

\begin{proposition}\label{prop:OddGENs}
If $G$ is an odd $\Gamma_1$ group, then $\gen(G)=*2$.
\end{proposition}

\begin{proof}
Every option of $\Phi(G)$ has type $t_2$ by Corollary~\ref{cor:OddType}, so $\otype(X_{\Phi(G)})=\{t_2\}$.  Then 
\[\type(X_{\Phi(G)})=(1,\mex\{0,1\},\mex\{2\})=(1,2,0),\] so $\gen(G)=*2$. 
\end{proof}

\begin{exam}
Since $H:=\ZZ_7 \rtimes \ZZ_3$ is in $\Gamma_1$, $\gen(H)=*2$ by Proposition~\ref{prop:OddGENs}.  Note that $H$ is the smallest odd $\Gamma_1$ group.    The smallest odd nonabelian group that is not $\Gamma_1$ is $K:=(\ZZ_3 \times \ZZ_3) \rtimes \ZZ_3$, yet $\gen(K)=*2$, too.  It is possible to get nim-numbers other than $*2$ for odd groups; in fact, $\gen(\ZZ_3 \times \ZZ_3 \times \ZZ_3)=*1$.  These three examples agree with~\cite[Corollary~4.8]{ErnstSieben}, which states that $\gen(G) \in \{*1,*2\}$ for odd nontrivial $G$. 
\end{exam}

\begin{proposition}\label{prop:AllMaximalsSameParityTypes}
Assume $G \in \Gamma_1$ and $I$ is an odd non-Frattini intersection subgroup of $G$.
If $I$ is only contained in odd (respectively, even) maximal subgroups, then $\type(X_I)$ is $t_2$ (respectively, $t_1$).
\end{proposition}

\begin{proof}
Since $I\not=\Phi(G)$ and $G \in \Gamma_1$, $t_0 \in \otype(X_I)$. In both cases of the proof, we use structural induction on the structure classes.

First, we assume that $I$ is only contained in odd maximal subgroups. 
If $X_J$ is a nonterminal option of $X_I$, then $J$ is also only contained in odd maximal subgroups. So $\type(X_J)=t_2$ by induction.
Hence $\{t_0\}\subseteq\otype(X_I)\subseteq\{t_0,t_2\}$, which implies $\type(X_I)=t_2$ using Proposition~\ref{prop:PossibleTypes}.

Next, we assume that $I$ is only contained in even maximal subgroups.  Since $I$ is odd, there is an involution $t \notin I$ such that $I\cup\{t\}$ is contained in an even maximal subgroup.  The structure class  $X_J$ containing $I \cup \{t\}$ is a type $t_4$ option of $X_I$. So we may conclude that $\{t_0,t_4\} \subseteq \otype(X_I)$.  
Let $X_J$ be a nonterminal option $X_I$. If $X_I$ is even, then $\type(X_J)=t_4$ by Table~\ref{table:typetable}. 
If $X_J$ is odd, then $\type(X_J)=t_1$ by induction, since $J$ is contained only in even maximal subgroups. 
Hence $\{t_0,t_4\}\subseteq\otype(X_I)\subseteq\{t_0,t_1,t_4\}$, which implies $\type(X_I)=t_1$ using Proposition~\ref{prop:PossibleTypes}.
\end{proof}

For the rest of the paper we will only consider even groups.

\begin{proposition}\label{prop:PossibleGENs}
If $G$ is an even $\Gamma_1$ group, then
\[
\gen(G) =  
   \begin{dcases} 
      *1, & t_2\notin\otype(X_{\Phi(G)}) \\
      *3, & t_2\in\otype(X_{\Phi(G)}), t_3\notin\otype(X_{\Phi(G)}) \\
      *4, & t_2,t_3\in\otype(X_{\Phi(G)}).
   \end{dcases}
\]
\end{proposition}

\begin{proof}
Since $G \in \Gamma_1$, $X_{\Phi(G)}$ has no option of type $t_0$.
Let $t$ be an involution of $G$ 
and $X_J$ be the structure class containing $\Phi(G) \cup \{t\}$.  Then $X_J$ is a type $t_4$ option of $X_{\Phi(G)}$. 
Hence $\{t_4\}\subseteq\otype(X_{\Phi(G)})\subseteq\{t_1,\ldots,t_4\}$. We compute $\type(X_{\Phi(G)})$ for every possibility for $\otype(X_{\Phi(G)})$ in Table~\ref{table:PossibleGENs}. 
The result follows from this calculation and the fact that $\gen(G)$ is equal to the second component of $\type(X_{\Phi(G)})$.
\end{proof}

\begin{table}
\begin{center}
\begin{tabular}{@{}lll@{}}
\toprule
$\otype(X_{\Phi(G)})$ & $\otype(X_{\Phi(G)})$ & $\type(X_{\Phi(G)})$ \\
\midrule
$\{t_4\}$ & $\{t_1,t_4\}$ &  $(1,1,0)$ \\  
$\{t_2,t_4\}$ & $\{t_1,t_2,t_4\}$ & $(1,3,0)$ \\ 
$\{t_3,t_4\}$ & $\{t_1,t_3,t_4\}$ &  $(1,1,0)$ \\
$\{t_2,t_3,t_4\}$ & $\{t_1,t_2,t_3,t_4\}$&  $(1,4,0)$ \\
\bottomrule
\end{tabular}
\end{center}

\vspace{0.1in}

\caption{\label{table:PossibleGENs} Spectrum of $\type(X_{\Phi(G)})$ for $G\in\Gamma_1$. 
Note that $\Phi(G)$ is the trivial group in this case.}
\end{table}

Recall that a subset $\mathcal{C}$ of the power set of a group $G$ is a \emph{covering} of $G$ if $\bigcup \mathcal{C}=G$; in this case, we also say that $\mathcal{C}$ \emph{covers} $G$ or $G$ \emph{is covered by} $\mathcal{C}$.

\begin{corollary}\label{cor:GeneralGENCorollary}
If a $\Gamma_1$ group $G$ is covered by the set of even maximal subgroups of $G$, then $\gen(G)=*1$.
\end{corollary}

\begin{proof}
We will demonstrate that $t_2 \notin \otype(X_{\Phi(G)})$, which will suffice by Proposition~\ref{prop:PossibleGENs}.
Let $X_I$ be an option of $X_{\Phi(G)}$. Then $I=\lceil g \rceil$ for some $g\in G$.
If $\lceil g\rceil$ is even, then $\type(X_{\lceil g \rceil}) = t_4\neq t_2$ by Proposition \ref{prop:PossibleTypes}. 
Now assume that $\lceil g\rceil$ is odd.  Since $G$ is covered by the set of even maximal subgroups, $g$ is contained in an even maximal subgroup $M$.  
By Cauchy's Theorem, there is an involution $t$ in $M$.  The structure class $X_J$ containing $\{g,t\}$ is a type $t_4$ option of $X_{\lceil g \rceil}$ by Proposition~\ref{prop:PossibleTypes}. 
So $\type(X_{\lceil g \rceil})$ 
cannot be $t_2$ again by Proposition~\ref{prop:PossibleTypes}.   
\end{proof}

If a noncyclic group $G$ has only even maximal subgroups, then the set of even maximal subgroups covers $G$.  The next corollary then follows immediately from Corollary~\ref{cor:GeneralGENCorollary}.

\begin{corollary}\label{cor:MainGENCorollary}
If $G$ is a $\Gamma_1$ group with only even maximal subgroups, then $\gen(G)=*1$.
\end{corollary}

\end{section}


\begin{section}{Symmetric Groups}\label{sec:Symmetric}


In light of Corollary~\ref{cor:MainGENCorollary}, we need only a simple proposition to completely determine the nim-numbers for the achievement and avoidance games for symmetric groups. 

\begin{proposition}\label{prop:SnMaximals}
If $n \geq 4$, then every maximal subgroup of $S_n$ has even order.
\end{proposition}

\begin{proof}
An odd subgroup cannot be maximal since it is contained in the even subgroup $A_n$.
\end{proof}

The nim-numbers for the  avoidance game for generating symmetric groups were computed in~\cite{ErnstSieben} 
for $n \in \{2,3,4\}$ and in \cite{BeneshErnstSiebenDNG} for $n \geq 5$.  
We include the statement of the result here for completeness.   Note that $S_1$ is trivial, so $\dng(S_1)$ does not exist.  

\begin{theorem}\label{thm:SymmetricDNG}
The values of $\dng(S_n)$ are
\[
\dng(S_n) =  
   \begin{dcases} 
      *1, & n=2  \\
      *3, & n=3 \\
      *0, & n \geq 4. 
   \end{dcases}
\]
\end{theorem}

We are now ready to determine the nim-numbers for $\gen(S_n)$.  This result verifies the portion of~\cite[Conjecture~9.1]{ErnstSieben} on symmetric groups.  

\begin{theorem}\label{thm:SymmetricGEN}
The values of $\gen(S_n)$ are
 \[
\gen(S_n) =  
   \begin{dcases} 
      *0, & n \in \{1,4\}\\
      *2, & n=2  \\
      *3, & n=3 \\
      *1, & n \geq 5. 
   \end{dcases}
\]
\end{theorem}

\begin{proof}
The empty set is a generating set for the trivial group, so $\gen(S_1)=*0$.  The cases where $n \in \{2,3,4\}$ were done in~\cite{ErnstSieben}, so assume $n \geq 5$.  
By~\cite{isaacs1995generating}, $S_n \in \Gamma_1$.  By Proposition~\ref{prop:SnMaximals}, every maximal subgroup of $S_n$ has even order.  
Hence $\gen(S_n)=*1$ by Corollary~\ref{cor:MainGENCorollary}.
\end{proof}

Theorems~\ref{thm:SymmetricDNG} and~\ref{thm:SymmetricGEN} immediately yield the following result.

\begin{corollary}
The first player has a winning strategy for
\begin{itemize}
\item $\dng(S_n)$ if and only if $n \in \{2,3\}$;
\item $\gen(S_n)$ if and only if $n \not\in \{1,4\}$. 
\end{itemize}
\end{corollary}
\end{section}


\begin{section}{Alternating Groups}\label{sec:Alternating}


Determining the nim-numbers of $\dng(A_n)$ and $\gen(A_n)$ requires much more background.  The following well-known proposition follows from Feit and Thompson's Odd Order Theorem~\cite{FeitThompson} and the fact that every group of order $2m$ for some odd $m$ has a normal subgroup of order $m$ (see~\cite[Corollary~6.12]{Isaacs1994} for example).

\begin{proposition}\label{prop:Simple4}
If $U$ is a nonabelian simple group, then $4$ divides $|U|$.
\end{proposition}

Below, we will make use of a special case of the O'Nan--Scott Theorem (see~\cite{liebeck1987classification} for instance).  
Recall that the \emph{general affine group} of degree $n$ over a field of size $p^k$ for a prime $p$
is defined to be the semidirect product $\AGL(n,p^k) := V \rtimes \GL(n,p^k) $ of a vector space $V$ of size $p^{nk}$ (i.e., dimension $n$) and the general linear group, 
where the latter acts on $V$ by linear transformations.

\begin{theorem}[O'Nan--Scott]\label{thm:ONanScott}
Let $A_n$ act on a set $\Omega$ of size $n$.  If $M$ is a maximal subgroup of $A_n$, then $M$ must be one of the following:
\begin{enumerate}
\item (Intransitive Case)  $M=(S_m \times S_k) \cap A_n$ with $n=m+k$;
\item (Imprimitive Case) $M=(S_m \wr S_k) \cap A_n$ with $n=mk$, $m,k > 1$;
\item (Affine Case)  $M=\AGL(k,p) \cap A_n$ with $n=p^k$ and $p$ prime;
\item (Diagonal Case)  $M=(U^k.(\operatorname{Out}(U) \times S_k)) \cap A_n$ with $U$ a nonabelian simple group, $k \geq 2$, and $n=|U|^{k-1}$;
\item (Wreath Case)  $M=(S_m \wr S_k) \cap A_n$ with $n=m^k$, $m \geq 5$, $k > 1$, and either $k \not=2$ or $m \not\equiv_4 2$; 
\item (Almost Simple Case) $U \lhd M \leq \operatorname{Aut}(U)$, where $U$ is a nonabelian simple group and $M$ acts primitively on $\Omega$. 
\end{enumerate}
\end{theorem}

\begin{rem}\label{rem:OrdersOfElementsOfAGL}
We will see that $\AGL(1,p)$ plays an important role in determining the nim-numbers for alternating groups.  
Note that $\GL(1,p) \cong \ZZ_{p}^{\times}$, where $\ZZ_p^{\times}$ is the group of units of $\ZZ_{p}$, which is cyclic of order $p-1$. Then $\AGL(1,p) \cong \ZZ_p \rtimes \ZZ_p^{\times}$, 
where the action is field multiplication.  It is also easy to check that every element of $\AGL(1,p)$ either has order $p$ or has order dividing $p-1$.  
\end{rem}

\begin{corollary}\label{cor:AnMaximalOrders}
Let $n \geq 5$.  If $A_n$ has an odd maximal subgroup, then $n$ is prime, $n \equiv_4 3$, and every odd maximal subgroup of $A_n$ is isomorphic to $\AGL(1,n) \cap A_n$, and hence has order $\frac{1}{2}n(n-1)$.
\end{corollary}

\begin{proof}
If $H$ is a subgroup of $S_n$ that is not contained in $A_n$, then $|H \cap A_n|=\frac{1}{2}|H|$.  Then $H \cap A_n$ is even if $4$ divides $|H|$.  By inspection, $S_m \times S_k$ and $S_m \wr S_k$ have orders that are divisible by $4$ under the conditions specified in the Intransitive, Imprimitive, and Wreath Cases of Theorem~\ref{thm:ONanScott}, so $A_n$ cannot have an odd maximal subgroup in those cases.  
Similarly, any $U^k.(\operatorname{Out}(U)) \times S_k)$ and any almost simple $M$ are divisible by $4$  by Proposition~\ref{prop:Simple4}, so $A_n$ cannot have an odd maximal subgroup in the Diagonal or Almost Simple Cases of Theorem~\ref{thm:ONanScott}. 

This leaves only the Affine Case to be considered. Assume that $n=p^k$ for some prime $p$, and let $M$ be a maximal subgroup of $A_n$ 
in the Affine case of Theorem~\ref{thm:ONanScott}.  The order of $\AGL(k,p)$ is $p^k(p^k-1)(p^k-p)(p^k-p^2)\cdots(p^k-p^{k-1})$, which is divisible by $4$ if $k \geq 2$. 
Then we may assume that $k=1$, so $p=n \geq 5$ and we conclude that $p$ is odd. Then $M \cong \AGL(1,p) \cap A_p \cong \left(\ZZ_p \rtimes \ZZ_p^{\times}\right) \cap A_p$ by Remark~\ref{rem:OrdersOfElementsOfAGL}. 
Since $A_p$ does not contain a $(p-1)$-cycle, we conclude that $|M|=\frac{1}{2}p(p-1)$, which is odd if and only if $p \equiv_4 3$. 
\end{proof}

The next result follows directly from Remark~\ref{rem:OrdersOfElementsOfAGL} and Corollary~\ref{cor:AnMaximalOrders}. 

\begin{corollary}\label{cor:ElementsOfAGL}
Let $n \geq 5$.  If $A_n$ has an odd maximal subgroup $M$, then every nontrivial element of $M$ is either an $n$-cycle or a power of a product of two $\frac{1}{2}(n-1)$-cycles.
\end{corollary}

\begin{proposition}\label{prop:AnOddMaximals}
Let $n \geq 5$.  Then $A_n$ has an odd maximal subgroup if and only if $n$ is prime, $n\equiv_4 3$, and $n \not\in \{7,11,23\}$.      
\end{proposition}

\begin{proof}
Suppose that $A_n$ has an odd maximal subgroup $M$.  Then by Corollary~\ref{cor:AnMaximalOrders}, $n$ is prime, $n \equiv_4 3$, and $M \cong \AGL(1,n) \cap A_n$.    
The subgroup $M \cong \AGL(1,n) \cap A_n$ is not maximal if $n \in \{7,11,23\}$ by the main theorem from ~\cite{liebeck1987classification}.

Thus, we may assume that $n$ is prime, $n \equiv_4 3$, and $n \not\in \{7,11,23\}$.  Again by the main theorem from~\cite{liebeck1987classification}, we have that $\AGL(1,n) \cap A_n$ is maximal in $A_n$.  Its order is $\frac{1}{2}n(n-1)$ by Corollary~\ref{cor:AnMaximalOrders}, which is odd because $n \equiv_4 3$.
\end{proof}

Recall~\cite{Dubner} that a prime $p$ is a \emph{generalized repunit prime} if there is a prime-power $q$ and integer $n$ such that $p=(q^n-1)/(q-1)$.   The next definition will simplify the statement of the proposition that follows.

\begin{definition}
A prime $p$ is said to be a \emph{$\zeta$-prime} if all of the following conditions hold:
\begin{enumerate}
\item $p \equiv_4 3$;
\item\label{item:exceptions} $p \notin \{11,23\}$;
\item $p$ is not a generalized repunit prime.
\end{enumerate}
\end{definition}

The $\zeta$-primes that are less than $100$ are $19, 43, 47, 59, 67, 71, 79$, and $83$~\cite{BeneshErnstSiebenZetaOEIS}.  Note that $7=111_2$ is a generalized repunit prime, so we did not have to explicitly exclude the prime $7$ from Condition~(\ref{item:exceptions}) to match the set of exceptions from Proposition~\ref{prop:AnOddMaximals}.  

\begin{proposition}\label{prop:AnEvensCoverCondition}
If $n \geq 5$, then the following are equivalent.
\begin{enumerate} 
\item\label{item:FailCover} The even maximal subgroups of $n$ fail to cover $A_n$.
\item\label{item:nCycles} There exists an $n$-cycle of $A_n$ that is not in any even maximal subgroup.
\item\label{item:ZetaPrime} $n$ is a $\zeta$-prime.
\end{enumerate} 
\end{proposition}

\begin{proof}
First, we show that Items~(\ref{item:FailCover}) and~(\ref{item:nCycles}) are equivalent.  
If there exists an $n$-cycle of $A_n$ that is not in any even maximal subgroup, then the even maximal subgroups of $A_n$ fail to cover $A_n$ by definition.  
Now suppose that the even maximal subgroups of $n$ fail to cover $A_n$.  
Then there must be an odd maximal subgroup of $A_n$, so $n$ is equal to some prime $p$ with $p \equiv_4 3$ and $p \not\in\{7,11,23\}$ by Proposition~\ref{prop:AnOddMaximals}.  
Let $r$ be the integer such that $p=2r+1$.  If $M$ is an odd maximal subgroup of $A_p$, then each element of $M$ is either trivial, 
a $p$-cycle, or a power of two disjoint $r$-cycles by Corollary~\ref{cor:ElementsOfAGL}.  
The identity and every power of two disjoint $r$-cycles is contained in an even subgroup isomorphic to $(S_r \times S_r) \cap A_p$, which is contained in some even maximal subgroup.  
Therefore, it must be a $p$-cycle that is not contained an even maximal subgroup. 

To finish, we show that Items~(\ref{item:nCycles}) and~(\ref{item:ZetaPrime}) are equivalent.  
Let $g$ be an $n$-cycle of $A_n$ that is not in any even maximal subgroup.  
Then $g$ must be contained in an odd maximal subgroup because $A_n$ is not cyclic. 
Proposition~\ref{prop:AnOddMaximals} implies that $n$ is a prime such that $n \equiv_4 3$ with $n \not\in\{7,11,23\}$.   
Theorem~\ref{thm:ONanScott} and the fact that $n$ is prime imply that any maximal subgroup containing $g$ must 
be isomorphic to $\operatorname{AGL}(1,n) \cap A_n$ or an almost simple group $H$ that acts primitively on a set of size $n$.   
The former has odd order, and~\cite[Table~3]{liebeck1985primitive} lists all the possibilities for $H$.   
Therefore, $g$ will be contained in an even maximal subgroup $H$ if and only if $H$  appears in~\cite[Table~3]{liebeck1985primitive}.  
The only rows in this table where the second column (labeled by $n$) is equal to the fourth column (labeled by $p$, which is equal to $n$ in our case) correspond to the 
first $\operatorname{PSL}(d,q)$, $\operatorname{Sz}(q)$, $M_{23}$, and $M_{11}$.  But the row for $\operatorname{PSL}(d,q)$ implies that $n$ is a generalized repunit prime, 
the row for $\operatorname{Sz}(q)$ implies that $n \equiv_4 1$, and the other two imply that $n \in \{7,11,23\}$.  So the $\zeta$-primes were defined exactly to exclude the entries in this table.  
Therefore, if $g$ is not contained in any even maximal subgroup, then $g$ is not in any $H$ listed in \cite[Table~3]{liebeck1985primitive} and hence $p$ is a $\zeta$-prime.  

Conversely, assume that $n$ is a $\zeta$-prime. Then $n$ is a prime such that $n \equiv_4 3$  and $n\notin\{7,11,23\}$.
So $A_n$ has no subgroup $H$ from \cite[Table~3]{liebeck1985primitive}, and hence $g$ is not 
contained in any even maximal subgroup, proving Item~(\ref{item:nCycles}).
\end{proof}


\begin{subsection}{Avoidance Games for Generating Alternating Groups}


We will use the following result to  determine the nim-numbers of $\dng(A_n)$.

\begin{proposition}\label{prop:DNGCriteria}
\cite[Corollary~6.4]{BeneshErnstSiebenDNG} Let $G$ be a nontrivial finite group.
\begin{enumerate}
\item If all maximal subgroups of $G$ are odd, then $\dng(G)=*1$.
\item If all maximal subgroups of $G$ are even, then $\dng(G)=*0$.
\item Assume $G$ has both even and odd maximal subgroups.
  \begin{enumerate}
 \setlength{\itemindent}{-14pt}
  \item If the set of even maximal subgroups covers $G$, then $\dng(G)=*0$.
  \item If the set of even maximal subgroups does not cover $G$, then $\dng(G)=*3$.
  \end{enumerate}
\end{enumerate}
\end{proposition}

\begin{theorem}\label{thm:AlternatingDNG}
The values of $\dng(A_n)$ are \[
\dng(A_n) =  
   \begin{dcases} 
      *3, & n \in \{3,4\} \text{ or $n$ is a $\zeta$-prime}\\
      *0, & \text{otherwise}.
   \end{dcases}
\]
\end{theorem}

\begin{proof}
The cases where $n \in \{3,4\}$ were done in~\cite{ErnstSieben}, so assume $n \geq 5$.  
If $n$ is not a $\zeta$-prime, then the set of even maximal subgroups covers $A_n$ by Proposition~\ref{prop:AnEvensCoverCondition};  in this case, $\dng(A_n)=*0$ by Proposition~\ref{prop:DNGCriteria}.  If $n$ is a $\zeta$-prime, then the set of even maximal subgroups fails to cover $A_n$ by Proposition~\ref{prop:AnEvensCoverCondition}.  This implies that $A_n$ has an odd maximal subgroup. The group $A_n$ contains the proper subgroup  $\langle (1,2)(3,4) \rangle$ of order $2$, and so $A_n$ must also contain an even maximal subgroup.  We may conclude that $\dng(A_n)=*3$ if $n$ is a $\zeta$-prime by Proposition~\ref{prop:DNGCriteria}.  
\end{proof}

Just like for $S_1$,  $A_1$ and $A_2$ are trivial, so $\dng(A_1)$ and $\dng(A_2)$ do not exist. 

\end{subsection}


\begin{subsection}{Achievement Games for Generating Alternating Groups}


We will see that $\zeta$-primes play an important role in determining the nim-numbers of $\gen(A_n)$ as they did for $\dng(A_n)$.  The following theorem refutes the portion of~\cite[Conjecture~9.1]{ErnstSieben} on alternating groups. 

\begin{theorem}\label{thm:AlternatingGEN}
The values of $\gen(A_n)$ are 
\[
\gen(A_n)  =  
   \begin{dcases}
      *0, & n \in \{1,2\} \\
      *2, & n=3 \\
      *3, & n=4  \\ 
      *4, & \text{$n$ is a $\zeta$-prime} \\ 
      *1, & \text{otherwise}. 
   \end{dcases}
\]
\end{theorem}

\begin{proof} 
The empty set is a generating set for the trivial group, so $\gen(A_1)=*0=\gen(A_2)$. The cases where $n \in \{3,4\}$ were done in~\cite{ErnstSieben}, so assume $n \geq 5$.  By~\cite{chigira1997generating}, $A_n \in \Gamma_1$.  If every maximal subgroup of $A_n$ has even order, then $\gen(A_n)=*1$ by Corollary~\ref{cor:MainGENCorollary}.    We only need to determine what happens in the case that $A_n$ has an odd maximal subgroup.  
Then $n=p$ for some prime $p \not\in \{7,11,23\}$ such that $p \equiv_4 3$ by Proposition~\ref{prop:AnOddMaximals}. We may write $p=2r+1$ for some odd $r$.  If $p$ is not a $\zeta$-prime, then the even maximal subgroups cover $A_p$ by Proposition~\ref{prop:AnEvensCoverCondition}, so $\gen(A_p)=*1$ by Corollary~\ref{cor:GeneralGENCorollary}. 

So assume that $p$ is a $\zeta$-prime.  Then there is an odd maximal subgroup $M$ of $A_p$; we know that $M$ must be isomorphic to $\AGL(1,p) \cap A_p$ by Corollary~\ref{cor:AnMaximalOrders}.  
Then $M=\langle g,x \rangle$ where $g$ is a $p$-cycle and $x$ is the product of two $r$-cycles and each element of $M$ is either trivial, a $p$-cycle, or a power of two disjoint $r$-cycles by Corollary~\ref{cor:ElementsOfAGL}.  By Proposition~\ref{prop:AllMaximalsSameParityTypes}, $\type(M)=t_2$. 
 
The element $x$ is contained in an even subgroup isomorphic to $(S_r \times S_r) \cap A_p$, so $X_{\lceil x \rceil}$ 
has an even option of type $t_4$.  The structure class $X_{\lceil  x \rceil}$ has the type $t_2$ option $X_M$, since $\{x,g\} \in X_M$.  Proposition~\ref{prop:PossibleTypes} implies $\type(X_{\lceil  x  \rceil})=t_3$. 

Since $p$ is a $\zeta$-prime, there is a $p$-cycle $y$ that is not contained in any even maximal subgroup by Proposition~\ref{prop:AnEvensCoverCondition}.  Since all $p$-cycles are conjugate in $S_p$, we conclude that $g$ is also only contained in odd maximal subgroups, so $\type(X_{\lceil g \rceil})=t_2$ by Proposition~\ref{prop:AllMaximalsSameParityTypes}.  Then $t_2,t_3 \in \otype(X_{\Phi(A_p)})$, so $\gen(A_p)=*4$ by Proposition~\ref{prop:PossibleGENs}.
\end{proof}

\end{subsection}


\begin{subsection}{Outcomes for Alternating Groups}


Theorems~\ref{thm:AlternatingDNG} and~\ref{thm:AlternatingGEN} immediately yield the following result.
\begin{corollary}
The first player has a winning strategy for
   \begin{itemize}
\item $\dng(A_n)$ if and only if $n \in \{3,4\}$ or $n$ is a $\zeta$-prime;
\item $\gen(A_n)$ if and only if $n \not\in \{1,2\}$. 
   \end{itemize}
\end{corollary}

\end{subsection}

\end{section}


\begin{section}{Further Questions}


We conclude with a few open problems.

\begin{enumerate}
\item Recall that every finite simple group is in $\Gamma_1$ by~\cite{Guralnick2000}.  It is well-known that many finite simple groups have the property that every maximal subgroup has even order (see~\cite{BeneshErnstSiebenDNG},~\cite{Atlas}, and~\cite{WilsonBook}). For such $G$, $\dng(G)=*0$ by the main results from~\cite{BeneshErnstSiebenDNG} and $\gen(G)=*1$ by Corollary~\ref{cor:MainGENCorollary}.  Can we determine the nim-numbers for $\dng(G)$ and $\gen(G)$ for every finite simple group $G$?   
\item Can we determine the nim-numbers for $\dng(G)$ and $\gen(G)$ if $G$ is almost simple?
\item Can the conditions from Proposition~\ref{prop:PossibleGENs} be translated into group-theoretic conditions?  For instance, can one describe when $t_2 \in \otype(X_{\Phi(G)})$ based on the subgroup structure of $G$? 
\end{enumerate}

\end{section}


\bibliographystyle{amsplain}
\bibliography{game}


\end{document}